\documentclass[12pt,a4paper]{amsart}

\usepackage[english]{babel}
\usepackage[T1]{fontenc}
\usepackage[utf8]{inputenc}

\usepackage{amsmath, amsthm, amscd, amsfonts}
\usepackage{amssymb}
\usepackage{mathtools}

\usepackage{hyperref}
\usepackage[capitalise]{cleveref}

\usepackage{setspace}
\usepackage[cmtip,arrow]{xy}
\usepackage{pb-diagram,pb-xy}
\usepackage{comment}
\newcommand{\RR}{{\mathbb{R}}}

\newcommand{\Rforce}{\mathbb{R}}    % The second of the forcings used.

\newcommand{\GCH}{GCH}

\DeclareMathOperator{\len}{lh}

\DeclareMathOperator{\crit}{crit}
\DeclareMathOperator{\dom}{dom}
\DeclareMathOperator{\cf}{cf}
\DeclareMathOperator{\Col}{Col}
\DeclareMathOperator{\Add}{Add}

\DeclareMathOperator{\Ult}{Ult}

\DeclareMathOperator{\stem}{stem}

\DeclareMathOperator{\range}{range}

\def\MPB{{\mathbb{P}}}

\def\MRB{{\mathbb{R}}}

\def\k{\kappa}

\def\l{\lambda}

\def\a{\alpha}

\newtheorem{theorem}{Theorem}[section]
\newtheorem{lemma}[theorem]{Lemma}

\newtheorem{definition}[theorem]{Definition}

\newtheorem{remark}[theorem]{Remark}
\newtheorem{claim}[theorem]{Claim}
\newtheorem{question}[theorem]{Question}

\numberwithin{equation}{section}

\usepackage{pb-diagram}

\def\l{\lambda}

\def\rmark{\mbox{$\rm\bf\rule{0.06em}{1.45ex}\kern-0.05em R$}}
\def\pmark{\mbox{$\rm\bf\rule{0.06em}{1.45ex}\kern-0.05em P$}}
\def\nmark{\mbox{$\rm\bf\rule{0.06em}{1.45ex}\kern-0.05em N$}}
\def\vdash{\mbox{$\rm\| \kern-0.13em -$}}

\def\l{\lambda}

\def\rmark{\mbox{$\rm\bf\rule{0.06em}{1.45ex}\kern-0.05em R$}}
\def\pmark{\mbox{$\rm\bf\rule{0.06em}{1.45ex}\kern-0.05em P$}}
\def\nmark{\mbox{$\rm\bf\rule{0.06em}{1.45ex}\kern-0.05em N$}}
\def\vdash{\mbox{$\rm\| \kern-0.13em -$}}

\title[(Weak) Diamond can fail at the least inaccessible cardinal]{(Weak) Diamond can fail at the least inaccessible cardinal}

\author[M. Golshani]{Mohammad Golshani}
\address{School of Mathematics\\ Institute for Research in Fundamental Sciences (IPM)\\
P.O. Box:
19395-5746,\\
Tehran-Iran.}
\email{golshani.m@gmail.com}

\date{}
\thanks{The author's research has been supported by a grant from IPM (No. 98030417).}

\begin{document}
\begin{abstract}
Starting from suitable large cardinals, we force the failure of (weak) diamond  at the least inaccessible cardinal. The result improves an unpublished theorem of Woodin
and a recent result of Ben-Neria, Garti and Hayut.
\end{abstract}
\subjclass[2020]{03E35; -3E55}

\keywords{Diamond, inaccessible cardinal, Radin forcing}

\thanks{ } \maketitle

\section{Introduction}
We study the combinatorial principles diamond, introduced by Jensen \cite{jensen}, and weak diamond, introduced by Devlin-Shelah \cite{devlin-shelah},
and prove the consistency of their failure at the least inaccessible cardinal.

Suppose  $\kappa$ is an uncountable regular cardinal. Recall that  diamond at $\kappa,$ denoted $\Diamond_\kappa$, asserts the existence of a  sequence $\langle S_\alpha \mid \alpha < \kappa    \rangle$
such that
for each $\alpha < \kappa, S_\alpha \subseteq \alpha$, and if $X \subseteq \kappa,$ then $\{\alpha < \kappa \mid X \cap \alpha = S_\alpha   \}$ is stationary in $\kappa.$

Also, the weak diamond at $\kappa,$ denoted $\Phi_\kappa$, is the assertion ``For every $c: 2^{<\kappa} \to 2$, there exists $g: \kappa \to 2$
such that for all $f: \kappa \to 2,$ the set $\{\alpha < \kappa \mid c(f \upharpoonright \alpha) = g(\alpha)      \}$
is stationary in $\kappa$''.

It is easily seen that $\Diamond_{\kappa^+}$ implies $2^\kappa=\kappa^+$, and in fact by a celebrated theorem of Shelah \cite{shelah}, for all uncountable cardinals $\kappa, \Diamond_{\kappa^+}$
is equivalent to $2^\kappa=\kappa^+$. It follows that it is easy to force the failure of diamond at successor cardinals.

By \cite{devlin-shelah}, $2^\kappa < 2^{\kappa^+}$ implies $\Phi_{\kappa^+}$. It was later observed by Abraham and Baumgartner that  $\Phi_{\kappa^+}$ imples $2^\kappa < 2^{\kappa^+}$, see for example \cite{garti}, \cite{shelah-proper} where the proof is attributed to Abraham, or \cite{taylor} where the proof is attributed to Baumgartner. It follows that $\Phi_{\kappa^+}$ and $2^\kappa < 2^{\kappa^+}$ are equivalent and that  $\Diamond_{\kappa^+}$ implies $\Phi_{\kappa^+}$.  It also shows that it is easy to force $\Phi_{\kappa^+}$ to fail.

Unlike the case of successor cardinals, it is difficult to force the failure of diamond or weak diamond at an inaccessible cardinal.
By an old unpublished result of Woodin \cite{cummings99},  if $2^\kappa > \kappa^+$
and if $u$ is a measure sequence on $\kappa$ of length $\kappa^+,$
then in the generic extension by Radin forcing using $u$, diamond fails at $\kappa.$
A recent result of Ben-Neria, Garti and Hayut shows that if the model is prepaired more carefully, then the weak diamond at $\kappa$
fails in the resulting model as well.

However in the above model, $\kappa$ is not small, in the sense that it is a limit of some very large cardinals.
In this paper we extend the above results of Woodin and Ben-Neria, Garti and Hayut to obtain the failure of (weak) diamond at the least inaccessible cardinal.

\begin{theorem}  \label{thm:main theorem}
Assume $\kappa$ is a $(\kappa+3)$-strong cardinal. Then there is a generic extension of the universe in which
$\kappa$ is the least inaccessible cardinal and $\Phi_\kappa$ (and hence also $\Diamond_\kappa$) fails.
\end{theorem}
\begin{remark}
	\begin{enumerate}
		\item In Theorem \ref{thm:main theorem}, we can make $\kappa$ to be the least Mahlo cardinal or the least greatly Mahlo cardinal and so on. See remarks after Question \ref{diamond and gch}.
		
	\item By \cite{zeman}, some large cardinal assumptions are needed to get the failure of diamond at Mahlo cardinals.	
	\end{enumerate}

\end{remark}
In Section \ref{sec:radin forcing} we define the generic extension we are looking for and prove some of its basic properties.
Then in Section \ref{sec:diamond} we extend Woodin's theorem by showing that $\Diamond_\kappa$ fails in the resulting model. Finally in Section
\ref{sec:weak diamond}, we extend the Ben-Neria, Garti and Hayut result and show that $\Phi_\kappa$ also fails in the model.
We may mention that, though a model for the failure of $\Phi_\kappa$ is necessarily a model in which $\Diamond_\kappa$ fails,  we have decided to include the proofs of the failure of both $\Diamond_\kappa$  and $\Phi_\kappa$, as we found both of the proofs interesting. Also the proof of the failure of
$\Diamond_\kappa$ is based on the unpublished work of Woodin \cite{cummings99} and so it will make public the basic ideas first developed by Woodin in connection to the failure of diamond at inaccessible cardinals.

\section{Radin forcing with interleaved collapses}
\label{sec:radin forcing}
In this section we define the main forcing notion used in the proof of  \ref{thm:main theorem} which is based on ideas from \cite{cummings}. We assume that the ground model satisfies the following:
\begin{itemize}
\item $\kappa$ is $(\kappa+2)$-strong, as witnessed by $j: V \to M \supseteq V_{\kappa+2}$ with $\crit(j)=\kappa.$

\item For each inaccessible cardinal $\alpha \leq \kappa, 2^\alpha=\alpha^{++}$, further, if $0<n<\omega,$ then $2^{\alpha^{+n}} = \alpha^{+n+1}$.

\item $M\models$``$2^\kappa=\kappa^{++}$ and for all $0<n < \omega, 2^{\kappa^{+n}} = \kappa^{+n+1}$''.

\item There is $F \in V$ which is $\MPB=\Col(\kappa^{+4}, < i(\kappa))_N$-generic over $N$, where $i: V \to N $ is the ultrapower of
$V$ by $U=\{ X \subseteq \kappa: \kappa \in j(X)       \}$.

\item $j$ is generated by a $(\kappa, \kappa^{+3})$-extender.
\end{itemize}

Such a model can be constructed starting from $GCH$ and a  $(\kappa+3)$-strong cardinal $\kappa$ \cite{cummings}.
Note that
$F \in M$, as it can be coded by an element of $V_{\kappa+2}$, also, if $k: N \to M$ is the induced elementary embedding, defined by
$k([f]_{U})=j(f)(\kappa),$
then $\crit(k)=\kappa^{+3}_N < \kappa^{+3}_M=\kappa^{+3}$
and $F$ can be transferred along $k$, in the sense that $\langle k''[F] \rangle$, the filter generated by $k''[F]$, is $\Col(\kappa^4, < j(\kappa))_M$-generic over $M$.

Set

$\hspace{1.5cm}$ $P^*=\{  f: \kappa \to V_\kappa \mid \dom(f) \in U$ and $\forall \alpha, f(\alpha) \in    \Col(\alpha^{+4}, < \kappa)                \}$.

$\hspace{1.5cm}$ $F^* = \{ f \in P^* \mid i(f)(\kappa) \in F    \}$.

Then $U$ can be read off $F^*$ as $U= \{ X \subseteq \kappa \mid \exists f \in F^*, X=\dom(f)    \}$.

\subsection{Measure sequences}
The following definitions are based on \cite{cummings} with  modifications required for our purposes.
\begin{definition}
A constructing pair, is a pair $(k, g)$ where
\begin{itemize}
\item $k: V \to k(V)$ is a non-trivial elementary embedding into  a transitive inner model, and if $\alpha=\crit(k),$ then
$^{\alpha}k(V)\subseteq k(V).$

\item $g$ is $\Col(\alpha^{+4}, < k_0(\alpha))_{k_0(V)}$-generic over $k_0(V)$, where $k_0: V \to k_0(V) \simeq \Ult(V, U_k)$ is the ultrapower embedding approximating $k$. Also factor $k$ through $k_0$, say $k=l \circ k_0.$

\item $g \in k(V)$.

\item If we transfer $g$  along $l$ to get $\langle l''[g] \rangle$, the filter generated by $l''[g]$,  then it is  $\Col(\alpha^{+4}, < k(\alpha))_{k(V)}$-generic over $k(V)$.
\end{itemize}
\end{definition}
In particular note that the pair $(j, F)$ constructed above is a constructing pair.
\begin{definition}
If $(k, g)$ is a constructing pair as above, then $g^* = \{ f \in P^* \mid k_0(f)(\alpha) \in g    \}.$
\end{definition}
\begin{definition}
Suppose $(k, g)$ is a constructing pair as above. A sequence $w$ is constructed by $(k, g)$ iff
\begin{itemize}
\item $w \in k(V).$
\item $w(0) = \alpha = \crit(k)$.
\item $w(1)=g^*.$
\item For $1 < \beta< \len(w), w(\beta)=\{ X \subseteq V_\alpha \mid w \upharpoonright \beta \in k(X)              \}$.
\item $k(V) \models |\len(w)| \leq w(0)^{+}$.
\end{itemize}
\end{definition}
If $w$ is constructed by $(k, g)$, then we set $\kappa_w=w(0),$
and if $\len(w) \geq 2$, then we define

$\hspace{1.5cm}$ $F^*_w=w(1)$.

$\hspace{1.5cm}$ $\mu_w = \{ X \subseteq \kappa_w \mid \exists f \in F^*_w, X = \dom(f)     \}$.

$\hspace{1.5cm}$ $\bar{\mu}_w = \{ X \subseteq V_{\kappa_w} \mid \{ \alpha \mid \langle \alpha \rangle \in \mu_w     \} \in \mu_w       \}$.

$\hspace{1.5cm}$ $F_w= \{ [f]_{\mu_w} \mid f \in F^*_w   \}$.

$\hspace{1.5cm}$ $\mathcal{F}_w = \bar{\mu}_w \cap \bigcap \{w(\alpha) \mid 1 < \alpha < \len(w)          \}$.

\begin{definition}
Define inductively

$\hspace{1.5cm}$ $\mathcal{U}_0 = \{ w \mid \exists (k, g)$ such that $(k, g)$  constructs $w  \}$.

$\hspace{1.5cm}$ $\mathcal{U}_{n+1}= \{ w \in \mathcal{U}_n \mid \mathcal{U}_n \cap V_{\kappa_w} \in  \mathcal{F}_w         \}.$

$\hspace{1.5cm}$ $\mathcal{U}_{\infty}= \bigcap_{n \in \omega} \mathcal{U}_{n}$.

The elements of $\mathcal{U}_{\infty}$ are called measure sequences.
\end{definition}
Now let $u$ be the measure sequence constructed using $(j, F)$ above. It is easily seen that for each $\alpha < \kappa^{++}, u \upharpoonright \alpha$
exists and is in $\mathcal{U}_{\infty}$.

In the next subsection, we assign to each $w \in \mathcal{U}_{\infty}$
a forcing notion $\Rforce_w,$ which is Radin forcing with interleaved collapses. The forcing notion used for the proof of \ref{thm:main theorem} is then $\Rforce_{u \upharpoonright \kappa^+}$.

\subsection{Definition of forcing}
We are now ready to define our main forcing notion. The forcing is very similar to the one defined by Cummings  \cite{cummings}. We assign to each measure sequence  $w \in \mathcal{U}_{\infty}$, a forcing notion $\Rforce_w$.The forcing $\Rforce_w$ adds a club $C$ of ground model regular cardinals into $\kappa_w$ in such a way that if $\alpha < \beta$ are successive points in $C$, then it collapses all cardinals in the interval $(\alpha^{+4}, \beta)$ into $\alpha^{+4}$
and makes $\beta=\alpha^{+5}$. When $\len(w)=\kappa_w^+,$ then the forcing preserves the inaccessibility of $\kappa_w$ and makes it the least
inaccessible cardinal above $\min(C)$.
\begin{definition}
Assume $w$ is a measure sequence. $\mathbb{P}_w$ consists of all tuples $s= (w, \lambda, A, H, h),$ where
\begin{enumerate}
\item $w$ is a measure sequence.

\item $\lambda < \kappa_w.$

\item $A \in \mathcal{F}_w$.

\item $H \in F^*_w$ with  $\dom(H)=\{ \kappa_v > \lambda \mid v \in A  \}$.

\item $h \in \Col(\lambda^{+4}, < \kappa_w).$
\end{enumerate}

Note that if $\len(w)=1,$ then the above tuple is of the form  $s=(w, \lambda, \emptyset, \emptyset, h)$ (where  $\lambda < \kappa_w$ and $h \in \Col(\lambda^{+4}, < \kappa_w)).$
\end{definition}
We define the order $\leq^*$ on $\MPB_w$ as follows.
\begin{definition}
Assume $s= (w, \lambda, A, H, h)$ and $s'= (w', \lambda', A', H', h')$ are in $\MPB_w.$ Then $s \leq^* s'$ iff
\begin{enumerate}
\item $w=w'$ and $\lambda=\lambda'.$
\item $A \subseteq A'.$
\item For all $u \in A, H(\kappa_u) \supseteq H'(\kappa_u).$
\item $h \supseteq h'.$
\end{enumerate}
\end{definition}
We now define our main forcing notion.
\begin{definition}
If $w$ is a measure sequence, then $\Rforce_w$ is the set of finite sequences $p= \langle  p_k \mid k \leq n     \rangle$,
where
\begin{enumerate}
\item $p_k = (w_k, \lambda_k, A_k, H_k, h_k) \in \MPB_{w_k}$ , for each $k \leq n.$

\item $w_n=w.$

\item If $k<n,$ then $\lambda_{k+1}=\kappa_{w_k}$.
\end{enumerate}
\end{definition}
Given $p \in \Rforce_w$ as above, we denote it by
\[
p= \langle  p_k \mid k \leq n^p     \rangle
\]
and call $n^p$ the length of $p$.
We also use $w^p_k$ for $w^{p_k}$ and so on (for $k \leq n^p$).
The direct extension relation $\leq^*$ is defined  on $\Rforce_w$  in the natural way:
\begin{definition}
Assume $p= \langle  p_k \mid k \leq n     \rangle$ and $p'= \langle  p'_k \mid k \leq n'     \rangle$ are in $\Rforce_w.$
Then $p \leq^* p'$ ($p$ is a direct extension or a Prikry extension of $p'$) iff $n=n'$ and for all
$k \leq n, p_k \leq^* p'_k.$
\end{definition}
We now define one point extension of a condition, which allows us to define the extension relation on $\Rforce_w.$
\begin{definition}
Assume $p= (w, \lambda, A, H, h) \in \MPB_w$  and $w' \in A.$ Then $\Add(p, w')$ is the condition $\langle p_0, p_1 \rangle \in \Rforce_w$ defined by
\begin{enumerate}
\item $p_0 = (w', \lambda, A \cap V_{\kappa_{w'}}, H \upharpoonright V_{\kappa_{w'}}, h)$.

\item $p_1= (w, \kappa_{w'}, A \setminus V_\eta, H \upharpoonright \dom(H) \setminus V_\eta, H(\kappa_{w'})  ),$
where $\eta = \sup \range (H(\kappa_{w'}))$.
\end{enumerate}
 In the case that this does not yield a member of $\Rforce_w$, then $\Add(p,w')$ is undefined.

If $p=\langle p_0,\dots,p_n \rangle\in \Rforce_w$ and $u\in A_k$ for some $k\leq n$ then $\Add(p,u)$ is the member of $\Rforce_w$ obtained by replacing $p_k$ with the two members of $\Add(p_k,u)$.
That is,
\begin{itemize}
	\item  $\Add(p,u) \upharpoonright  k=p \upharpoonright k.$
	\item  $\Add(p,u)_{k}=\Add(p_k,u)_{0}.$
	\item  $\Add(p,u)_{k+1}=\Add(p_k,u)_{1}.$
	\item  $\Add(p,u) \upharpoonright [k+2,n+1]= p \upharpoonright [k+1,n].$
\end{itemize}

The forcing order $\leq$ on $\Rforce_w$ is the smallest transitive relation containing the direct order $\leq^*$ and all pairs of the form $(p,\Add(p,u))$.
\end{definition}
The next lemma shows that any condition $p$  has a direct extension $q$ such that
$\Add(q, u)$ is well-defined for all $k \leq n^q$ and all $u \in A^q_k$.
\begin{lemma}
	\label{almost all can be added}
	\begin{itemize}
		\item [(a)] Suppose $p= (w, \lambda, A, H, h) \in \MPB_w$. Then
		\[
		A'=\{  w' \in A:  \Add(p, w') \text{~is well-defined~}           \} \in \mathcal{F}_w.
		\]
		\item [(b)] Suppose $p\in \Rforce_w.$ Then there exists
		$q \leq^* p$ such that for all $k \leq n^q$ and all $u \in A^q_k, \Add(q, u) \in \Rforce_w$ is well-defined.
	\end{itemize}
\end{lemma}
\begin{proof}
	Clause (b)   follows from (a), so let us prove (a). Let $p= (w, \lambda, A, H, h) \in \MPB_w$. We have to show that
	$A' \in \bar{\mu}_w \cap \bigcap \{w(\alpha) \mid 1 < \alpha < \len(w)          \}.$
	
	Suppose $(k, g)$ constructs $w$, where $k: V \to M,$ and let $k_0: V \to N \simeq \Ult(V, U)$
	be the corresponding ultrapower embedding.
	
	Let us first show that $A' \in \bar{\mu}_w.$ We have
	
	$\hspace{2.cm}$ $A' \in \bar{\mu}_w \iff \{ \a: \langle \a \rangle  \in A'           \} \in U$
	
	$\hspace{3.3cm}$ $\iff \kappa_w \in k(\{ \a: \langle \a \rangle  \in A'           \})$
	
	$\hspace{3.3cm}$ $\iff \langle \kappa_w \rangle  \in k(A')$
	
	$\hspace{3.3cm}$ $\iff \Add(k(p), \langle \kappa_w \rangle)$ is a well-defined condition in $\Rforce^M_{k(w)}.$
	
	But we have
	$$ \Add(k(p), \langle \kappa_w \rangle)= \langle (\langle \kappa_w \rangle, \lambda, A, H, h), (k(w), \kappa_w, A^*, H^*,  k(H)(\kappa_w))                        \rangle,$$
	where $A^*=k(A) \setminus V_{\kappa_w}$ and $H^*= k(H) \upharpoonright \dom(k(H)) \setminus V_{\kappa_w}$. Thus $\Add(k(p), \langle \kappa_w \rangle)$   is  well-defined,
	which implies $A' \in \bar{\mu}_w.$
	
	Now let  $1< \a < \len(w).$  Then
	
	$\hspace{2.cm}$ $A' \in w(\a) \iff w \upharpoonright \a \in k(A')$
	
	$\hspace{3.65cm}$ $\iff \Add(k(p), w \upharpoonright \a)$ is a well-defined condition in $\Rforce^M_{k(w)}.$
	
	By an argument as above, it is easily seen that $\Add(k(p), w \upharpoonright \a)$ is a well-defined condition in $\Rforce^M_{k(w)},$ and hence $A' \in w(\a)$.
	Thus $$A' \in \bar{\mu}_w \cap \bigcap \{w(\alpha) \mid 1 < \alpha < \len(w)          \}$$ as required.
\end{proof}
By the above lemma we can always assume that $\Add(q, u)$ is well-defined for all $q \in \Rforce_w$, all $k \leq n^q$ and all $u \in A^q_k$.

\subsection{Basic properties of the forcing notion $\Rforce_{w}$}
We now state the main properties of the forcing notion $\Rforce_w$.
\begin{lemma} \label{thm:chain condition}
$(\Rforce_w, \leq)$ satisfies the $\kappa_w^+$-c.c.
\end{lemma}
\begin{proof}
Assume on the contrary that $A \subseteq \Rforce_w$ is an antichain of size $\kappa_w^+.$ We can assume that all $p \in A$ have the same length $n$.
Write each $p \in A$ as $p = d_p ^{\frown} p_n,$ where $d_p \in V_{\kappa_w}$ and $p_n = (w, \lambda^p, A^p, H^p, h^p)$. By shrinking $A$, if necessary,
we can assume that there are fixed $d \in V_{\kappa_w}$ and $\lambda < \kappa_w$ such that for all $p \in A,$ $d_p=d$ and $\lambda_p = \lambda.$

Note that for $p \neq q$ in $A$, as $p$ and $q$ are incompatible, we must have $h^p$ is incompatible with $h^q$. But $\Col(\lambda^{+4}, < \kappa_w)$
satisfies the $\kappa_w$-c.c., and we get a contradiction.
\end{proof}
\begin{lemma} (The factorization lemma)  \label{thm:factorization lemma}
Assume $p=\langle p_0,\dots,p_n \rangle\in \Rforce_w$ with $p_i=(w_i, \lambda_i, A_i, H_i, h_i)$ and $m<n.$ Set $p^{\leq m}= \langle p_0, \dots, p_m \rangle$
and $p^{>m} = \langle  p_{m+1}, \dots, p_n     \rangle$.

$($a$)$  $p^{\leq m} \in \Rforce_{w_m},$
$p^{>m} \in \Rforce_{w}$ and there exists $i: \Rforce_w / p \rightarrow \Rforce_{w_m}/p^{\leq m} \times \Rforce_w / p^{>m}$
which is an isomorphism with respect to both $\leq^*$ and $\leq.$

$($b$)$ If $m+1 <n,$ then there exists
$i: \Rforce_w / p \rightarrow \Rforce_{w_m}/p^{\leq m} \times \Col(\kappa_{w_m}^{+4}, <\kappa_{w_{m+1}}) \times \Rforce_w / p^{>m+1}$
which is an isomorphism with respect to both $\leq^*$ and $\leq.$ \hfill$\Box$
\end{lemma}
The next lemma can be proved as in \cite{cummings} (in fact in \cite{cummings}  the lemma is proved for a more complicated forcing notion).
\begin{lemma}  \label{thm:prikry property}
  $(\RR_w, \leq, \leq^*)$ satisfies the Prikry property, i.e., given a condition $p \in \RR_w$ and a statement $\sigma$ of the forcing language, there exists $q \leq^* p$ which decides $\sigma.$
\end{lemma}
\begin{proof}
As requested by the referee, we provide a proof for completeness. We follow  \cite{hayut-eskew}.
	We prove the lemma by induction on $\k_w.$ Thus, assuming  it is true for $\Rforce_u$ with $\k_u < \k_w,$ we prove it for $\Rforce_w.$ We prove the lemma in a sequence of claims.

Let us first show that it suffices to prove the lemma for conditions of length $1$.
\begin{claim}
	Suppose the lemma holds for all conditions of length $1$. Then it holds for all conditions
\end{claim}
\begin{proof}
 We prove the claim  by induction on $\len(p)$. When $\len(p)=1$, this follows from the assumption.
	Thus suppose that $\len(p) \geq 2$; say
	$$p= s^{\frown} \langle (u, \lambda', A', H', h') \rangle ^{\frown} \langle  (w, \lambda, A, H, h)          \rangle.$$
	By the factorization Lemma \ref{thm:factorization lemma}, we have
	\[
	\Rforce_w/p \simeq (\Rforce_u / s^{\frown} \langle (u, \lambda', A', H', h') \rangle) \times (\Rforce_w / \langle  (w, \lambda, A, H, h)          \rangle).
	\]
	Let $\langle  s_i: i < \k_u     \rangle$
	enumerate $\mathbf{L} \cap V_{\k_u}$, where $\mathbf{L}$ is the set of all stems of conditions in $\Rforce_w$. We define by recursion on $i$ a $\leq^*$-decreasing chain $\langle p_i: i \leq \k_u    \rangle$
	of conditions in $\Rforce_w / \langle  (w, \lambda, A, H, h)          \rangle$
	as follows:
	
	Set $p_0=\langle  (w, \lambda, A, H, h)          \rangle.$ Given $p_i,$ let $p_{i+1} \leq^* p_i$
	decide whether there is a condition in $\Rforce_u / s^{\frown} \langle (u, \lambda', A', H', h') \rangle$
	with stem $s_i$ which decides $\sigma$ and if so, then it forces one of $\sigma$
	or $\neg \sigma.$
	At limit ordinals $i \leq \k_u,$  let $p_i$  $\leq^*$-extend all
	$p_j, j<i.$ Such a $p_i$ exists as $(\Rforce_w / \langle  (w, \lambda, A, H, h)          \rangle, \leq^*)$
	is $\k_w$-closed.
	
	By our construction,
	\begin{center}
		$\Vdash_{\Rforce_u / s^{\frown} \langle (u, \lambda', A', H', h') \rangle}$``$p_{\k_u}$ decides $\sigma$''.
	\end{center}
	By the induction hypothesis, there exists $q \leq^* s^{\frown} \langle (u, \lambda', A', H', h') \rangle$ which decides which way
	$p_{\k_u}$ decides $\sigma,$ and then $q^{\frown} p_{\k_u} \leq^* p$ decides $\sigma.$
\end{proof}
So by the above claim we are reduced to the case where $\len(p)=1$. Thus let $p=(w, \lambda, A, H, h) \in \MPB_w$.

	Given $q \in \Rforce_w,$ it can be written as $q=d_q^{\frown} \langle q_{\len(q)}  \rangle$, where $d_q \in V_{\k_w}$ and $q_{\len(q)} \in \MPB_w$.
	We set $\stem(q)=d_q$ and call it the stem of $q$.
	Let
	$\mathbf{L}$
	be the set of the stems of conditions in $\Rforce_w$ which extend $p$:
	\[
	\mathbf{L}=\{\stem(q): q \in \Rforce_w \text{~and~} q \leq p                  \}.
	\]
	
	Fix $v \in A$. Let $\mathbf{L}_v$ be the set of all stems $s=\langle q_k: k \leq n  \rangle \in \mathbf{L}$, with $q_k=(w_k, \lambda_k, A_k, H_k, h_k)$ (for $k \leq n$),  such that
	for some $A_v, H_v, h_v, A'$, $H'$ and $h'$,
	$$q=s^{\frown} \langle v, \lambda, A_v, H_v, h_v \rangle ^{\frown} \langle w, \k_v, A', H', h'    \rangle \in \Rforce_w$$
	and
	$$q \leq \Add(p, v).$$
	\begin{claim}
	$\mathbf{L}_v$ has size less than $\lambda^{+4}$.
	\end{claim}
	\begin{proof}
	Given a stem $s=\langle q_k: k \leq n  \rangle$, if
		$q \leq \Add(p, v)$ witnesses $s \in \mathbf{L}_v$, then
		 $\k_{q_n}=\k_{w^q_{n}}=\lambda$; from which the result follows immediately.
	\end{proof}
	For each $v \in A$ and each stem $s \in \mathbf{L}_v,$ define the sets $D^{\text{top}}(0, s, v)$ and $D^{\text{top}}(1, s, v)$,
	as follows:
	\begin{itemize}
		\item
		$D^{\text{top}}(0, s, v)$ is the set of all $g \leq H(\k_v)$ for which there exist  $A_v, H_v, h_v, A'$ and $H'$ such that
		$s^{\frown} \langle v, \lambda, A_v, H_v, h_v \rangle ^{\frown} \langle w, \k_v, A', H', g    \rangle$ extends $ \Add(p, v)$ and
		decides $\sigma.$
		
		\item $D^{\text{top}}(1, s, v)$ is the set of all $g \leq H(\k_v)$ such that for all $A_v, H_v, h_v, A', H'$ and $g' \leq g,$
		$s^{\frown} \langle v, \lambda, A_v, H_v, h_v \rangle ^{\frown} \langle w, \k_v, A', H', g'    \rangle$
		does not decide $\sigma.$

	\end{itemize}
	Clearly, $D^{\text{top}}(0, s, v) \cup D^{\text{top}}(1, s, v)$ is dense in $\Col(\k_v^{+4}, < \k_w)/H(\k_v)$, and so by the distributivity of
	$\Col(\k_v^{+4}, < \k_w),$ the intersection
	\[
	D^{\text{top}}_v= \bigcap_{s \in \mathbf{L}_v} (D^{\text{top}}(0, s, v) \cup D^{\text{top}}(1, s, v))
	\]
	is also dense in $\Col(\k_v^{+4}, < \k_w)/H(\k_v)$.
	Take $\tilde H \in F^*_w$ such that
	\[
	\tilde A=\{ v \in A: \tilde H(\k_v) \in D^{\text{top}}_v            \} \in \mathcal{F}_w.
	\]
	Let $H^* \in F^*_w$ extend both of $H$ and $\tilde H.$
	
	Next define the sets
	$D^{\text{low}}(0, s, v)$ and $D^{\text{low}}(1, s, v)$
	as follows:
	\begin{itemize}
		\item $D^{\text{low}}(0, s, v)$ is the set of all $g \leq h$ for which there exist  $A_v, H_v, A'$ and $H'$ such that
		$s^{\frown} \langle v, \lambda, A_v, H_v, g \rangle ^{\frown} \langle w, \k_v, A', H', H^*(\k_v)    \rangle$ extends $\Add(p, v)$ and
		decides $\sigma.$
		
		\item $D^{\text{low}}(1, s, v)$ is the set of all $g \leq h$ such that for all $A_v, H_v,  A', H'$ and $g' \leq H^*(\k_v),$
		$s^{\frown} \langle v, \lambda, A_v, H_v, g \rangle ^{\frown} \langle w, \k_v, A', H',g'    \rangle$
		does not decide $\sigma.$
	\end{itemize}
	The set $D^{\text{low}}(0, s, v) \cup D^{\text{low}}(1, s, v)$ is dense in $\Col(\l^{+4}, < \k_v)/h$, and hence by the distributivity of
	$\Col(\l^{+4}, < \k_v)$, the intersection
	\[
	D^{\text{low}}_v= \bigcap_{s \in \mathbf{L}_v} (D^{\text{low}}(0, s, v) \cup D^{\text{low}}(1, s, v))
	\]
	is also dense in $\Col(\l^{+4}, < \k_v)/h$.
	Take $ \tilde h_v \in D^{\text{low}}_v.$

	Now consider $$p'=(w, \lambda, \tilde A, H^*, h)\leq p.$$
	For any stem
	$s$ of a condition in $\Rforce_w$ extending $p'$ and every $\alpha < \len(w),$
	let $A(s, \alpha) \in \mathcal{F}_w$ be such that one of the following three possibilities hold for it:
	\begin{enumerate}
		\item [($1_{s, \alpha}$):] For every $v \in A(s, \alpha)$ there exists $q' \leq p'$ such that  $q'$ forces $\sigma$
		and $q'$ is of the form
		$$q'= s^{\frown} \langle   v, \lambda, A'_{s, v}, H'_{s, v}, h'_{s, v}      \rangle ^{\frown} \langle w, \kappa_v, A_{s, v},  H_{s, v}, h_{s,v}     \rangle,$$ for some $A'_{s, v}, H'_{s, v}, h'_{s, v} \leq \tilde h_v, A_{s, v}, H_{s, v}$ and $h_{s, v} \leq H^*(\k_v).$

		\item [($2_{s, \alpha}$):]  For every $v \in A(s, \alpha)$ there exists $q' \leq p'$ such that  $q'$ forces $\neg \sigma$
		and $q'$ is of the above form.
		%$$q= s^{\frown} \langle   v, \lambda, A'_{s, v}, H'_{s, v}, h'_{s, v}      \rangle ^{\frown} \langle w, \kappa_v, A_{s, v},  H_{s, v}, h_{s,v}     \rangle$$, for some %$A'_{s, v}, H'_{s, v}, h'_{s, v}, A_{s, v}, H_{s, v}, h_{s, v}.$

		\item [($3_{s, \alpha}$):] For every $v \in A(s, \alpha)$ there does not exist $q' \leq p'$ of the  above form  such that $q'$ decides  $\sigma.$
	\end{enumerate}
	For every $v$,  we may suppose that $H_{s, v}, h_{s, v}, H'_{s, v}$ and $h'_{s, v}$'s depend only on $v$, and so we denote them by $H_v, h_v, H'_v$ and $h'_v$ respectively. For each $\alpha$
	let $A(\alpha)=  \bigtriangleup_{s} A(s, \alpha)$ be the diagonal intersection of the $A(s, \alpha)$'s and set
	\[
	A^*= A' \cap \bigcup_{\alpha < \len(w)} A(\alpha) \in \mathcal{F}_w.
	\]
	Also let $$p^*= (w, \lambda, A^*, H^*, h).$$
	\begin{claim}
	If $v \in A' \cap A(s, \alpha)$
	and  one of the   $(1_{s, \alpha})$
	or $(2_{s, \alpha})$ happen, then we may take $h_v=H^*(\k_v)$ and $h'_v=\tilde h_v$.
	\end{claim}
	\begin{proof}
	If one of these possibilities happen, then
	$H^*(\k_v) \in D(0, s, v)$, so there are
	$\tilde A_v, \tilde H_v,  \tilde A'$ and $\tilde H'$ such that
	$$q=s^{\frown} \langle v, \lambda, \tilde A_v, \tilde H_v, \tilde h_v \rangle ^{\frown} \langle w, \k_v, \tilde A', \tilde H', H^*(\k_v)    \rangle\leq \Add(p, v)$$
	and
	$q$ decides $\sigma.$
	On the other hand, there exists
	\[
	q'= s^{\frown} \langle   v, \lambda, A'_{s, v}, H'_{s}, h'_{s}      \rangle ^{\frown} \langle w, \kappa_v, A_{s, v},  H_{v}, h_{v}     \rangle \leq p'
	\]
	which  decides $\sigma$ as well.
	%As $q' \leq p',$ we have $h_v \leq H^*(\k_v)$ and $h'_{s, v} \leq h$,
	But the conditions $q$
	and $q'$ are compatible and they decide the same truth value; hence we can take $h_v=H^*(\k_v)$ and $h'_v=\tilde h_v$.
	\end{proof}
The next claim completes the proof of the lemma.
\begin{claim}
There exists a direct extension of $p^*$ which decides $\sigma.$
\end{claim}	
	\begin{proof}
		Assume not and let $r \leq p^*$ be of minimal length which decides $\sigma,$ say it forces $\sigma$.  Let us write
	\[
	\stem(r)=s^{\frown} \langle u, \lambda, A^r, H^r, h^r     \rangle,
	\]
	where $s \in V_{\k_u}.$
	By our assumption, there exists $\alpha < \len(w)$
	such that $A(s, \alpha) \in w(\alpha)$ satisfies
	$(1_{s, \alpha})$,  so for every $v \in A(s, \alpha)$,
	there exists
	$q'_v \leq p'$ such that  $q'_v$ forces $\sigma$
	and $q'_v$ is of the form
	$$q'_v= s^{\frown} \langle   v, \lambda, A'_{v}, H'_{v}, \tilde h_{v}      \rangle ^{\frown} \langle w, \kappa_v, A_{v},  H_{v}, H^*(\k_v)     \rangle,$$ for some $A'_{v}, H'_{v}, , A_{v}$ and $H_{v}.$
	
	We show that there exists $q^* \leq p^*$ with $\stem(q^*)=s$
	such that every extension of $q^*$  is compatible with  $q'_v,$ for some $v \in A(s, \alpha).$
	This property implies
	that $q^*$ forces $\sigma$, contradicting the minimal choice of $\len(r)$. We note that by the
	definition of extension in the forcing $\Rforce_w$, we may assume from this point on that $s$ is empty.
	
	Consider the map $\phi: A(\langle \rangle, \alpha) \to V$ which is defined by
	$$\phi: v \mapsto (\phi_0(v), \phi_1(v))=(A_\nu, H_\nu).$$
	As $A(\langle \rangle, \alpha) \in w(\alpha)$, we have $w \upharpoonright \alpha \in j(A(\langle \rangle, \alpha))$ (where $j$ is the constructing embedding for $w$).
	Let
	\[
	(A^{<\alpha}, H^{<\alpha}) = j(\phi)(w \upharpoonright \alpha).
	\]
	%Then $A^{< \alpha} \in \bigcap_{\beta<\alpha}w(\beta)$ and $H^{<\alpha} \in F^*_w$.
	Also let
	\[
	A^\alpha=\{v \in A(\langle \rangle, \alpha): A^{<\alpha} \cap V_{\k_v}=A_v \text{~and~} H^{<\alpha} \upharpoonright V_{\k_v}=H_v                   \}
	\]
	and
	%By easy reflection arguments, $A^\alpha \in w(\alpha)$.
	\[
	A^{>\alpha} = \bigcup_{\alpha < \beta < \len(w)}\{v \in A^*:  A^\alpha \cap V_{\k_v} \in w(\beta)               \}.
	\]
	Then $A^{**}=A^{<\alpha} \cup A^\alpha \cup A^{>\alpha} \in \mathcal{F}_w$.
	Set $H^{**}=H^{<\alpha} \cap H^*$ and finally set
	\[
	q^*= \langle  w, \lambda, A^{**}, H^{**}, h      \rangle \leq p^*.
	\]
	We show that $q^*$ is as required. Thus let
	$$q= \langle (w_k, \lambda_k, A_k, H_k, h_k): k \leq n                  \rangle$$
	be an extension of $q^*.$ There are various cases:
	\begin{enumerate}
		\item There is no index $k$ such that $\len(u_k)>0$ and $(A^\alpha \cup A^{>\alpha}) \cap V_{\k_{w_k}} \in \bigcup_{\beta < \len(w_k)} w_k(\beta).$
		Then pick some non-trivial measure sequence $v \in A^\alpha \cap A_n,$ and note that for all $k<n,$
		$A^{<\alpha} \cap A_k \in \bigcap_{\beta < \len(w_k)}w_k(\beta)$. Then one can easily show that $q$ is compatible with $q'_v$.

		\item There is an index $k$ with $\len(u_k)>0$ and $(A^\alpha \cup A^{>\alpha}) \cap V_{\k_{w_k}} \in \bigcup_{\beta < \len(w_k)} w_k(\beta)$
		and $A^\alpha \in w_k(\beta)$ for some $\beta < \len(w_k).$
		Let us pick $k$ to be the least such an index.
		Let $v \in A_k$ be such that $A^{<\alpha} \cap A_k \in \bigcap_{\beta < \len(v)} v(\beta)$.
		Then $q$ is compatible with $q'_v$.
		
		\item There is an index $k$ with $\len(u_k)>0$ and $(A^\alpha \cup A^{>\alpha}) \cap V_{\k_{w_k}} \in \bigcup_{\beta < \len(w_k)} w_k(\beta)$
		and $A^{>\alpha} \in w_k(\beta)$ for some $\beta < \len(w_k).$
		Then by our choice of $A^{>\alpha},$ there is some $v \in A_k$
		that can be added to $q$ such that we reduce to the  case (2).
	\end{enumerate}
	\end{proof}
	The lemma follows.
\end{proof}

Now suppose that $w=u \upharpoonright \kappa^+$, where $u$ is the measure sequence constructed by the pair $(j, F)$ and let $G \subseteq \Rforce_w$ be generic over $V$. Set
\begin{center}
$C=\{\kappa_u \mid \exists p \in G, \exists i < \len(p), p_i= (u, \lambda, A, H, h)  \}.$
\end{center}
By standard arguments, $C$ is a club of $\kappa.$
Let $\langle \kappa_i: i<\kappa \rangle$ be the increasing enumeration of the club $C$ and let $\vec{u}=\langle u_i \mid i<\kappa \rangle$
be the enumeration of $\{u \mid \exists p \in G, \exists i< \len(p), p_i= (u, \lambda, A,H, h) \}$ such that for $i<j<\kappa, \kappa_{u_i}=\kappa_i < \kappa_j=\kappa_{u_j}.$ Also let $\vec{F}=\langle F_i \mid i<\kappa \rangle$ be such that each $F_i$ is $\Col(\kappa_i^{+4}, <\kappa_{i+1})$-generic over $V$ produced by $G.$
 \begin{lemma}\label{thm: bounding sets}
 \begin{enumerate}
\item [(a)] $V[G]=V[\vec{u}, \vec{F}].$

\item [(b)] For every limit ordinal $j<\kappa, \langle \vec{u}\upharpoonright j, \vec{F}\upharpoonright j  \rangle$  is $\RR_{u_j}$-generic over $V$,
and $\langle \vec{u} \upharpoonright [j, \kappa), \vec{F}\upharpoonright$
 $[j, \kappa) \rangle$ is $\RR_w$-generic over $V[\vec{u}\upharpoonright j, \vec{F}\upharpoonright j ].$

\item [(c)] For every $\gamma <\kappa$ and every $A \subseteq \gamma$ with $A\in V[\vec{u}, \vec{F}],$ we have $A\in V[\vec{u}\upharpoonright j, \vec{F}\upharpoonright j],$
 where $j$ is the least ordinal such that $\gamma< \kappa_j.$
\end{enumerate}
\end{lemma}
We now state a geometric characterization of generic filters for $\MRB_{w}$. Such a characterization was first given by Mitchell \cite{mitchell82} for Radin forcing. The characterization given bellow is essentially due to Cummings \cite{cummings}.
\begin{lemma} (Geometric characterization)
	\label{geometric characterization}
	The pair $(\vec{u}, \vec{F})$ is $\MRB_{w}$-generic over $V$ if and only if it satisfies the following conditions:
	\begin{enumerate}
		\item If $\xi < \k$ and $\len(u_\xi)>1,$ then the pair $(\vec{u} \upharpoonright \xi, \vec{F} \upharpoonright \xi)$ is $\MRB_{u_\xi}$-generic over $V$.
		
		\item For all $A \in V_{\kappa+1}~(A \in \mathcal{F}_w \iff \exists \a<\kappa~ \forall \xi>\a,~ u_\xi \in A)$.
		
		\item For all $f \in w(1)$ there exists $\a < \kappa$ such that $ \forall \xi>\a, f(\k_\xi) \in F_\xi.$
	\end{enumerate}
\end{lemma}
\

As $\len(w)=\kappa^+,$ it follows from Mitchell \cite{mitchell82} (see also \cite{gitik}) that
\begin{lemma}  \label{thm:preserving inaccessibility}
In $V[G]$,	$\kappa$ becomes the least strongly inaccessible cardinal above $\kappa_0.$
\end{lemma}
\begin{proof}
	Let us first show that $\kappa$ remains inaccessible in $V[G]$. We follow Cummings \cite{cummings}.
	Suppose not and let $p \in \Rforce_w, \delta < \kappa$ and $\dot{f}$ be such that
	\begin{center}
		$p \Vdash$``$\dot{f}: \delta \to \k$ is cofinal''.
	\end{center}
	Let $\theta > \kappa$ be large enough regular such that $p, \dot{f}, w, \Rforce_w \in H(\theta)$
	and let $\mathbf{X} \prec H(\theta)$
	be such that
	\begin{enumerate}
		\item $p, \k^+, \dot{f}, w, \Rforce_w \in \mathbf{X}.$
		\item $V_{\k} \subseteq \mathbf{X}.$
		\item $^{<\k}$$\mathbf{X} \subseteq \mathbf{X}.$
		\item $|\mathbf{X}|=\k$.
	\end{enumerate}
	Let $\pi: \mathbf{X} \to \mathbf{N}$ be the Mostowski collapse of $\mathbf{X}$ onto a transitive model $\mathbf{N}.$ Note that
	$\pi \upharpoonright \mathbf{X} \cap V_{\k+1} =id \upharpoonright \mathbf{X} \cap V_{\k+1}.$

	Let $v=\pi(w)$ and $\beta=\pi(\k^+).$ Then
	\[
	\pi(\mathcal{F}_w) = \mathcal{F}_w \cap \mathbf{X} = \mathcal{F}_w \cap \mathbf{N}
	\]
	and
	\[
	\forall \alpha  \in \mathbf{X} \cap \k^+,~  \pi(v(\alpha))=v(\pi(\alpha)) = v(\alpha) \cap \mathbf{X} = w(\alpha  ) \cap \mathbf{N}.
	\]
	Let $\bar \beta = \sup (\mathbf{X} \cap \k^+) < \k^+.$ Using  Lemma \ref{geometric characterization},
	if $K$ is $\Rforce_{w \upharpoonright  \gamma}$-generic over $V$, where $\bar \beta \leq \gamma < \k^+$, then $K$ is $\pi(\Rforce_w)$-generic over $\mathbf{N}$.
	But note that for any limit ordinal $\gamma$ as above with $\cf(\gamma)< \k,$ we have
	\[
	\Vdash_{\Rforce_{w \upharpoonright  \gamma}} \text{``}\cf(\k)=\cf(\gamma)\text{''}.
	\]
	We get a contradiction and hence $\kappa$ remains regular. Clearly, $\kappa$ remains strong limit. Let us now show that it is the least inaccessible cardinal above $\kappa_0$.
	By Lemmas \ref{thm:chain condition}, \ref{thm:factorization lemma}
	and \ref{thm: bounding sets},
	\[
	CARD^{V[G]} \cap [\kappa_0, \kappa) = \lim(C) \cup \{\alpha, \alpha^+, \alpha^{++}, \alpha^{+3}, \alpha^{+4} \mid \alpha \in C     \}.
	\]
	As every limit point of $C$ is singular in $V[G],$\footnote{Given any limit ordinal $i< \kappa$,  $\len(u_i)< \kappa_i^+$. If $cf(\len(u_i)) < \kappa_i$ then $cf^{V[G]}(\kappa_i)=cf(\len(u_i)) < \kappa_i$, and if  $cf(\len(u_i))=\kappa_i$ then $\cf(\kappa_i)=\omega < \kappa_i$.}  $\kappa$ is the least inaccessible cardinal above $\kappa_0.$
\end{proof}

\section{Diamond at $\kappa$ in $V[G]$}
\label{sec:diamond}
In this section we show that $\Diamond_\kappa$ fails in $V[G]$. It is worth noting that to get this result, we do not need the whole assumptions about the ground model $V$ made at the beginning of Section \ref{sec:radin forcing}.

\begin{lemma} \label{thm:diamond in extension}
$\Diamond_\kappa$ fails in $V[G]$.
\end{lemma}
\begin{remark}
More generally, as the proof of Lemma \ref{thm:diamond in extension} shows, if $2^\kappa$ in the ground model is greater than the length of the measure sequence $w$, then $\Diamond_\kappa$ fails in the generic extension by $\Rforce_w$.
\end{remark}
\begin{proof}
As in Woodin \cite{cummings99}, we show that there is no diamond sequence in the extension which guesses every subset of $\kappa$ from  the ground model $V$.
Assume not and let $\langle  \dot{S}_\alpha \mid \alpha < \kappa  \rangle$ be an $\Rforce_w$-name for a diamond sequence at $\kappa$.

Let $p \in \Rforce_w$ and write it as $p= \vec{d}^{\frown} (w, \lambda, A, H , h),$ where $\vec{d} \in V_\kappa.$  We may further assume that for each
 $u \in A,$  $\Add(p,u)$ is defined.  Thus for each $u \in A$, we can form the condition
$\Add(p, u)= \vec{d}^{\frown} \langle p^u_0, p^u_1  \rangle \in \Rforce_w$, where
\[
p^u_0 = (u, \lambda, A \cap V_{\kappa_u}, H \upharpoonright V_{\kappa_u}, h)
\]
and
\[
p^u_1 = (w, \kappa_u, A \setminus  V_{\eta_u}, H \upharpoonright \dom(H) \setminus V_{\eta_u}, H(\kappa_u)),
\]
where $\eta_u = \sup\range (H(\kappa_u))$.

By the factorization lemma \ref{thm:factorization lemma},
\[
\Rforce_w / \Add(p, u) \simeq (\Rforce_u / \vec{d}^{\frown}p^u_0) \times (\Rforce_w / p^u_1).
\]
As forcing with $\Rforce_w / p^u_1$ does not add new subsets to $\kappa_{u},$, we can look at $\dot{S}_{\kappa_u}$
as an $\Rforce_w$-name for an $\Rforce_u$-name for a subset of $\kappa_u.$ So by the Prikry property \ref{thm:prikry property}, we can find
$q^u_1 = (w, \lambda, A_u, H_u, h_u) \leq^* p^u_1$ and an $\Rforce_u$-name $\tau_u$ such that $\vec{d}^{\frown} \langle p^u_0, q^u_1  \rangle$ forces
 $\dot{S}_{\kappa_u}$ to be the realization of $\tau_u$ by the generic up to the point $u$, i.e.,
 \[
 \vec{d}^{\frown} \langle p^u_0, q^u_1  \rangle \Vdash \text{~``} \dot{S}_{\kappa_u} = \tau_u \text{''}.
 \]
Let $A_1= \bigtriangleup_{u \in A} A_u = \{w' \in \mathcal{U}_\infty \cap V_\kappa \mid \forall u \in A,~ \kappa_u < \kappa_{w'} \Rightarrow w' \in A_u    \} \in \mathcal{F}_w$
be the diagonal intersection of $A_u$'s, $u \in A$.

Now consider the sequence $\langle  [H_u]_{\mu_w}: u \in A_1      \rangle \subseteq F_w$. Let $H^* \in F^*_w$ be such that $[H^*]_{\mu_w} \leq [H_u]_{\mu_w}$  for all
$u \in A_1.$
 Thus for each $u \in A_1,$ we can find  a measure one set $B_u \in \mathcal{F}_w$ on which $H^*$ extends $H_u$. Set
$A_2 = A_1 \cap  \bigtriangleup_{u \in A_1} B_u.$

Now let $A^* \in \mathcal{F}_w, A^* \subseteq A_2$, be such that for all $w', w'' \in A^*, h_{w'}$
and $h_{w''}$ are compatible.
Set
\[
p^* =\vec{d}^{\frown} (w, \lambda, A^*, H^*, h) \leq^* p.
\]
Then $p^* \leq^* p$ and for every $u \in A^*,$ we have $\Add(p^*, u)$ and
\[
\Add(p^*, u) \Vdash  \text{~``} \dot{S}_{\kappa_u} = \tau_u \text{''}.
\]
Let $\tau$ be defined on $A^*$ so that $\tau: u \mapsto \tau_u.$ In $M$, consider the map
$j(\tau).$ Then for each $\alpha < \kappa^+,$ we will have
$j(\tau)_{u \upharpoonright \alpha}$ which is an $\Rforce_{u \upharpoonright \alpha}$-name for a member of $P(\kappa) \cap M = P(\kappa) \cap V.$

By the chain condition property \ref{thm:chain condition}, there are at most $\kappa$ possibilities for the value of this name, so all in all we have at most $\kappa^+$ possibilities for the values of $j(\tau)_{u \upharpoonright \alpha}$ for every $\alpha < \kappa^+$, and since $2^\kappa > \kappa^+,$
we can find $S \subseteq \kappa, S \in V,$ such that for all $\alpha < \kappa^+, \Vdash_{\Rforce_{u \upharpoonright \alpha}}$``$\check{S} \neq j(\tau)_{u \upharpoonright \alpha}$''.

Consider $j(p^*)= \vec{d} ^{\frown} (j(w), \lambda, j(A^*), j(H^*), j(h)).$ For each $\alpha < \kappa^+,$
$\Add(j(p^*), w \upharpoonright \alpha) \in \Rforce^M_{j(w)}$
and it forces that $j(\dot{S})_\kappa$ is the realization of $j(\tau)_{u \upharpoonright \alpha}$. So this condition forces that
$j(\dot{S})_\kappa \neq \check{S}$,  and since $S=j(S) \cap \kappa,$ it follows from {\L}o\'{s}'s theorem that  for each $0 < \alpha < \kappa^+,$
\[
A^{**} = \{ u \in A^* \mid \Add(p^*, u) \Vdash \text{~``~} \dot{S}_{\kappa_u} \neq \check{S} \cap \kappa_u                           \} \in w(\alpha).
\]
Let $p^{**}= \vec{d}^{\frown}(w, \lambda, A^{**}, H^* \upharpoonright A^{**}, h)$, where $H^* \upharpoonright A^{**} = H^* \upharpoonright \{ \kappa_u > \lambda \mid \kappa_u \in A^{**}      \}$.

Then
\begin{center}
$p^{**} \Vdash$``$\dot{C} \cap \{\alpha < \kappa: \dot{S}_\alpha = \check{S} \cap \alpha       \}$  is bounded in $\kappa$''.
\end{center}
 We get a contradiction and the lemma follows.
 \end{proof}

Force over $V[G]$ with $\Col(\omega, \kappa_0)$ and let $H$ be $\Col(\omega, \kappa_0)$-generic over $V[G]$. As the forcing is small,
we can easily show that
\begin{center}
$V[G][H] \models$``$\kappa$ is the least inaccessible cardinal and $\Diamond_\kappa$ fails''.
\end{center}
So in $V[G][H]$, $\Diamond_\kappa$ fails.

\section{Weak diamond at $\kappa$ in $V[G]$}
\label{sec:weak diamond}
In this section, we  improve the conclusion of the last section, by showing that in fact $\Phi_\kappa$
 fails in the model constructed above. It suffices to prove the following.
 \begin{lemma} \label{thm:weak diamond in extension}
$\Phi_\kappa$ fails in $V[G].$
\end{lemma}
\begin{remark}
Unlike the proof of Lemma \ref{thm:diamond in extension}, where we only needed  $V$ to satisfy $2^\kappa > \kappa^+$, we need the extra assumption $M \models$`` $2^\kappa=2^{\kappa^+}=\kappa^{++}$'' for the proof of Lemma
\ref{thm:weak diamond in extension}.
\end{remark}
\begin{proof}	
The proof follows ideas developed in \cite{hayut}. Note that in $M$, we have $2^\kappa=2^{\kappa^+}=\kappa^{++}$,
so in $V$, there exists a partial function $H: \kappa \to V_\kappa$ such that $\dom(H)=\{  \alpha < \kappa \mid 2^\alpha =2^{\alpha^+}=\alpha^{++} \}$ \footnote{In fact we have $\dom(H)=\{  \alpha < \kappa \mid \alpha$ is an inaccessible cardinal$ \}.$}
and for all $\alpha \in \dom(H),$
\[
H(\alpha): 2^\alpha \leftrightarrow (V_\alpha^{<\omega} \times \alpha \times P(V_\alpha) \times P(V_\alpha\times V_\alpha) \times P(V_\alpha))^{\alpha \times \alpha^+}
\]
is a bijection. By slightly abuse of notations set $H(\kappa)=j(H)(\kappa)$.

As in \cite{hayut}, let us identify a condition $p= \vec{d}^{\frown} (w, \lambda, A, H , h) \in \Rforce_w$ with
\[
\langle \vec{d},  \lambda, A, H , h  \rangle \in V_\kappa^{<\omega} \times \kappa \times P(V_\kappa) \times P(V_\kappa\times V_\kappa) \times P(V_\kappa)
\]
and call it a simple representation of $p$. As $\Rforce_w$ satisfies the $\kappa^{+}$-c.c., every antichain in $\Rforce_w$
can be represented as an element of $(V_\kappa^{<\omega} \times \kappa \times P(V_\kappa) \times P(V_\kappa\times V_\kappa) \times P(V_\kappa))^\kappa.$

Define $F: 2^{<\kappa} \to V_\kappa, F \in V$ as follows: for $t \in 2^{<\kappa}$ with $\dom(t)=\alpha, F(t) \in V_\kappa$
is a function such that
\begin{itemize}
\item $\dom(F(t)) = \{u \in \mathcal{U}_\infty \cap V_\kappa \mid \kappa_u=\alpha   \}$.
\item $F(t)(u) = \{ q \in \Rforce_u \mid q$ is simply represented by an element of $H(\alpha)(t)(\len(u))             \}$.
\end{itemize}
Note that $H(\alpha)(t) \in (V_\alpha^{<\omega} \times \alpha \times P(V_\alpha) \times P(V_\alpha\times V_\alpha) \times P(V_\alpha))^{\alpha \times \alpha^+},$
and so $F(t)(u)$ is well-defined.
Now define the coloring $c: 2^{<\kappa} \to 2, c \in V[G],$ as follows:
\begin{itemize}
\item If $t \in 2^{\kappa_i},$ for some $i<\kappa,$ then
\begin{center}
 $c(t)=$ $\left\{
\begin{array}{l}
         1  \hspace{1.6cm} \text{ if } F(t)(u_i) \cap (G_{u_i}) \neq \emptyset,\\
         0  \hspace{1.6cm} \text{ otherwise. }
     \end{array} \right.$
\end{center}
where $G _{u_i}$ is the $\Rforce_{u_i}$-generic filter generated by $G$ and $\kappa_i=\kappa_{u_i}$.

\item $c(t)=0$ for every other $t \in 2^{<\kappa}$.
\end{itemize}
We show that $c$ exemplifies the failure of $\Phi_\kappa$ in $V[G].$ Thus suppose that $g: \kappa \to 2, g\in V[G].$ We find some $f: \kappa \to 2, f\in V[G]$
such that the set $\{ \alpha < \kappa \mid c(f \upharpoonright \alpha) \neq g(\alpha)           \}$
contains a club of $\kappa.$

Let $\dot{g}$ be an $\Rforce_w$-name for $g$, and suppose  $p= \vec{d}^{\frown} (w, \lambda, A, H , h) \in \Rforce_w$.
As before, we can assume that for each
 $u \in A,$  $\Add(p,u)$ is defined.  Thus for each $u \in A$, we can form the condition
$\Add(p, u)= \vec{d}^{\frown} \langle p^u_0, p^u_1  \rangle \in \Rforce_w$, where
$p^u_0$ and $p^u_1$ are defined as in the proof of Lemma \ref{thm:diamond in extension}.
By the factorization lemma \ref{thm:factorization lemma},
\[
\Rforce_w / \Add(p, u) \simeq (\Rforce_u / \vec{d}^{\frown}p^u_0) \times (\Rforce_w / p^u_1).
\]
By the Prikry property \ref{thm:prikry property}, we can find
$q^u_1 = (w, \lambda, A_u, H_u, h_u) \leq^* p^u_1$ and an $\Rforce_u$-name $\sigma_u$ for an ordinal in $\{0,1\}$ such that
 \[
 \vec{d}^{\frown} \langle p^u_0, q^u_1  \rangle \Vdash \text{~``} \dot{g}(\kappa_u) = \sigma_u \text{''}.
 \]

As in the proof of Lemma \ref{thm:diamond in extension}, let $A_1= \bigtriangleup_{u \in A} A_u \in \mathcal{F}_w$
 be the diagonal intersection of $A_u$'s, $u \in A$.
 Consider the sequence $\langle  [H_u]_{\mu_w}: u \in A_1      \rangle \subseteq F_w$ and let $H^* \in F^*_w$ be such that $[H^*]_{\mu_w} \leq [H_u]_{\mu_w}$  for all
 $u \in A_1.$
 Thus for each $u \in A_1,$ we can find  a measure one set $B_u \in \mathcal{F}_w$ on which $H^*$ extends $H_u$. Set
 $A_2 = A_1 \cap  \bigtriangleup_{u \in A_1} B_u.$

 Now let $A^* \in \mathcal{F}_w, A^* \subseteq A_2$, be such that for all $w', w'' \in A^*, h_{w'}$
 and $h_{w''}$ are compatible, and set
 \[
 p^* =\vec{d}^{\frown} (w, \lambda, A^*, H^*, h) \leq^* p.
 \]
 Then $p^* \leq^* p$ and for every $u \in A^*,$ we have $\Add(p^*, u)$ and
 \[
 \Add(p^*, u) \Vdash  \text{~``} \dot{g}(\kappa_u) = \sigma_u \text{''}.
 \]
Consider the map $\sigma: u \mapsto \sigma_u$ which is defined on $A^*$.
In $M$, define the function $h: \kappa^+ \to (V_\kappa^{<\omega} \times \kappa \times P(V_\kappa) \times P(V_\kappa\times V_\kappa) \times P(V_\kappa))^\kappa$
as follows: for every $\tau < \kappa^+$ let $h(\tau) \in (V_\kappa^{<\omega} \times \kappa \times P(V_\kappa) \times P(V_\kappa\times V_\kappa) \times P(V_\kappa))^\kappa$ be a simple representation of a maximal antichain $A_\tau \subseteq \Rforce_{w \upharpoonright \tau}$
of conditions $q \in \Rforce_{w \upharpoonright \tau}$ which force $j(\sigma)_{w \upharpoonright \tau} =0.$
$h$ can be identified with an element of $(V_\kappa^{<\omega} \times \kappa \times P(V_\kappa) \times P(V_\kappa\times V_\kappa) \times P(V_\kappa))^{\kappa \times \kappa^+}$, and so we can consider the function
$f=H(\kappa)^{-1}(h): \kappa \to 2$, where $H(\kappa)=j(H)(\kappa)$.

Set $A^{**}= \{u \in A^* \mid F(f \upharpoonright \kappa_u)(u)$ is a maximal antichain of $\Rforce_u$ of conditions $q \Vdash$``$\sigma_u=0$''$ \}$.
Then we have $j(H)(\kappa)(j(f)\restriction \kappa)=H(\kappa)(f)=h$ and  $j(F)(j(f)\restriction \kappa)(w)=j(F)(f)(w)$ and
hence
 \begin{center}
 $j(F)(f)(w)$ is a maximal antichain of $\Rforce_w$ of conditions $q \Vdash$``$j(\dot{g})(\kappa)=0$''.
 \end{center}
It immediately follows that
 $A^{**} \in \mathcal{F}_w$. Set $p^{**}=\vec{d} ^{\frown} (w, \lambda, A^{**}, H^* \upharpoonright A^{**}, h)$.
  Then
\[
p^{**} \Vdash\text{~``} \dot{C} \cap \{\alpha < \kappa \mid c(f \upharpoonright \alpha) =\dot{g}(\alpha)          \} \text{~is bounded in~} \kappa\text{~''},
\]
where $\dot{C}$ is the canonical $\Rforce_w$-name for the Radin club. To see this note that if $G$ is $\Rforce_w$-generic over $V$ and $p^{**} \in G,$
then $C=\dot{C}[G]$ is almost contained in $\{\kappa_u: u \in A^{**}\}$ and for all $u \in A^{**}$ we have
$F(f \restriction \kappa_u)(u) \cap G_u \neq \emptyset,$ where $G_u=G \cap \Rforce_u,$
thus $c(f \restriction \kappa_u)=1$, while $g(\kappa_u)=\sigma_u=0$.
The result follows immediately.
\end{proof}
In the model of \ref{thm:main theorem}, $\GCH$ fails cofinally often below $\kappa$, and we do not know the answer to the following.
\begin{question}
	\label{diamond and gch}
Is it consistent with $\GCH$ that $\Diamond_\kappa$ (or $\Phi_\kappa$) fails for the least inaccessible cardinal.
\end{question}
Also it is possible to extend our result to make $\kappa$ the least Mahlo cardinal (by taking $\len(w)=\kappa^+ \cdot \kappa^+$), the least greatly Mahlo cardinal
 (by taking $\len(w)=(\kappa^+)^{ \kappa^+}$) and so on. On the other hand $\Diamond_\kappa$ (and hence $\Phi_\kappa$) holds if $\kappa$ is large enough, say a measurable cardinal. However the answer to the following is unknown.
\begin{question}
Is it consistent  that $\Diamond_\kappa$ (or $\Phi_\kappa$) fails for $\kappa$ a weakly compact cardinal.
\end{question}
\subsection*{Acknowledgements}
The author thanks Hugh Woodin \cite{woodin} for allowing him to use his unpublished ideas, and  James Cummings \cite{cummings2010} for explaining him Woodin's proof. He also thanks the referee of the paper for his/her many useful comments and corrections.

\end{document}